\documentclass[12pt,draft]{amsart}

\subjclass[2010]{Primary: 11N25, Secondary: 11N37}

\author{Paul Pollack}
\address{University of British Columbia\\ Department of Mathematics \\ 1984 Mathematics Road\\ Vancouver, British Columbia V6T 1Z2, Canada}

\address{Simon Fraser University\\Department of Mathematics\\ Burnaby, British Columbia V5A 1S6, Canada}

\email{pollack@math.ubc.ca}

\author{Lola Thompson}
\address{Dartmouth College\\Department of Mathematics\\6188 Kemeny Hall\\Hanover, New Hampshire 03755, United States}
\email{lauren.a.thompson@dartmouth.edu}
 
\title{Practical pretenders}
\usepackage{amsmath,amssymb,amsthm,geometry,url}
\DeclareMathAlphabet{\curly}{U}{rsfs}{m}{n}
\newtheorem{thm}{Theorem}[section]
\newtheorem{prop}[thm]{Proposition}

\newtheorem{cor}[thm]{Corollary}

\newtheorem{lem}[thm]{Lemma}
\theoremstyle{remark}
\theoremstyle{definition}

\newtheorem*{rmk}{Remark}

\begin{document}
\def\phi{\varphi}
\renewcommand{\labelenumi}{(\roman{enumi})}
\def\polhk#1{\setbox0=\hbox{#1}{\ooalign{\hidewidth
    \lower1.5ex\hbox{`}\hidewidth\crcr\unhbox0}}}
\newcommand{\del}{\ensuremath{\delta}}
\def\A{\curly{A}}
\def\B{\curly{B}}
\def\e{\mathrm{e}}
\def\E{\curly{E}}
\def\F{\mathbf{F}}
\def\C{\mathbf{C}}
\def\I{\curly{I}}
\def\N{\mathbf{N}}
\def\D{\curly{D}}
\def\Q{\mathbf{Q}}
\def\O{\curly{O}}
\def\V{\curly{V}}
\def\W{\curly{W}}
\def\Z{\mathbf{Z}}
\def\p{\tilde{p}}
\def\Pp{\curly{P}}
\def\pr{\mathfrak{p}}
\def\Proj{\mathbf{P}}
\def\q{\mathfrak{q}}
\def\Ss{\curly{S}}
\def\T{\curly{T}}
\def\Nm{\mathcal{N}}
\def\cont{\mathrm{cont}}
\def\ord{\mathrm{ord}}
\def\rad{\mathrm{rad}}
\def\lcm{\mathop{\mathrm{lcm}}}
\numberwithin{equation}{section}
\begin{abstract} Following Srinivasan, an integer $n\geq 1$ is called \emph{practical} if every natural number in $[1,n]$ can be written as a sum of distinct divisors of $n$. This motivates us to define $f(n)$ as the largest integer with the property that all of $1, 2, 3, \dots, f(n)$ can be written as a sum of distinct divisors of $n$. (Thus, $n$ is practical precisely when $f(n)\geq n$.) We think of $f(n)$ as measuring the ``practicality'' of $n$; large values of $f$ correspond to numbers $n$ which we term \emph{practical pretenders}. Our first theorem describes the distribution of these impostors: Uniformly for $4 \leq y \leq x$, 
\[ \#\{n\leq x: f(n)\geq y\} \asymp \frac{x}{\log{y}}. \]
This generalizes Saias's result that the count of practical numbers in $[1,x]$ is $\asymp \frac{x}{\log{x}}$.

Next, we investigate the maximal order of $f$ when restricted to non-practical inputs. Strengthening a theorem of Hausman and Shapiro, we show that every $n > 3$ for which
\[ f(n) \geq \sqrt{e^{\gamma} n\log\log{n}} \]
is a practical number.

Finally, we study the range of $f$. Call a number $m$ belonging to the range of $f$ an \emph{additive endpoint}. We show that for each fixed $A >0$ and $\epsilon > 0$, the number of additive endpoints in $[1,x]$ is eventually smaller than $x/(\log{x})^A$ but larger than $x^{1-\epsilon}$.
\end{abstract}
\maketitle

\section{Introduction}
In 1948, Srinivasan \cite{srinivasan48} initiated the study of \emph{practical numbers}, natural numbers $n$ with the property that each of $1, 2, 3, \dots, n-1$ admits an expression as a sum of distinct divisors of $n$. For example, every power of $2$ is practical (since every natural number admits a binary expansion), but there are many unrelated examples, such as $n=6$ or $n=150$. Srinivasan posed two problems: Classify all practical numbers and say something interesting about their distribution. 

The first of these tasks was carried to completion by Stewart \cite{stewart54} in 1954. The same classification was discovered independently, and almost concurrently, by Sierpi{\'n}ski \cite{sierpinski55}. Given a natural number $n$, write its canonical prime factorization in the form 
\begin{equation}\label{eq:nfact} n := p_1^{e_1} p_2^{e_2} \cdots p_r^{e_r},\quad\text{where}\quad p_1 < p_2 < \dots < p_r. \end{equation}
Put $n_0=1$, and for $1 \leq j \leq r$, put $n_j:=\prod_{i=1}^{j}p_i^{e_i}$. Then $n$ is practical if and only if
\begin{equation}\label{eq:ineqs} p_{j+1} \leq \sigma(n_j)+1 \quad \text{for all}\quad 0 \leq j < r. \end{equation}
Below, we refer to this as the \emph{Stewart--Sierpi\'nski classification} of practical numbers. This criterion implies, in particular, that all practical numbers $n > 1$ are even. 
Stewart and Sierpi\'nski also showed that if all of the inequalities \eqref{eq:ineqs} hold, then not only are all integers in $[1,n-1]$ expressible as a sum of distinct divisors of $n$, but the same holds for all integers in the longer interval $[1, \sigma(n)]$. Note that $[1,\sigma(n)]$ is the largest interval one could hope to represent, since the sum of all distinct divisors of $n$ is $\sigma(n)$.

The distribution of practical numbers has proved more recalcitrant. Let $PR(x)$ denote the count of practical numbers not exceeding $x$. Already in 1950, Erd\H{o}s \cite{erdos50} claimed he could show that the practical numbers have asymptotic density zero, i.e., that $PR(x)=o(x)$ as $x\to\infty$, but he gave no details. In 1984, Hausman and Shapiro \cite{HS84} made the more precise assertion that $PR(x) \leq x/(\log{x})^{\beta+o(1)}$, with $\beta = \frac{1}{2}(1-1/\log{2})^2 \approx 0.0979\ldots$. Their proof has an error (specifically, \cite[Lemma 3.2]{HS84} is incorrect); one should replace $\beta$ with the smaller exponent $1-\frac{1+\log\log{2}}{\log{2}} \approx 0.0860713$. Much sharper results on $PR(x)$ were soon established by Tenenbaum \cite{tenenbaum86, tenenbaum95P}, who proved that  $PR(x) = \frac{x}{\log{x}}(\log\log{x})^{O(1)}$. By a refinement of Tenenbaum's methods, Saias \cite{saias97} established in 1997 what is still the sharpest known result: There are absolute constants $c_1$ and $c_2$ with
\begin{equation}\label{eq:saias} c_1 \frac{x}{\log{x}} \leq PR(x) \leq c_2 \frac{x}{\log{x}} \quad\text{for all}\quad x\geq 2.\end{equation}
On the basis of the numerical data, Margenstern \cite{margenstern91} has conjectured that $\frac{PR(x)}{x/\log{x}}$ tends to a limit $\approx 1.341$. 

In this paper, we are concerned with what we term \emph{near-practical numbers} or \emph{practical pretenders}. Define $f(n)$ as the largest integer with the property that all of the numbers $1, 2, 3, \dots, f(n)$ can be written as a sum of distinct divisors of $n$. By definition, $n$ is practical precisely when $f(n) \geq n-1$. We define a \emph{near-practical} number as one for which $f(n)$ is ``large''. This definition is purposely vague; its nebulous nature suggests that we investigate the behavior of the two-parameter function 
\[ N(x,y):= \#\{n \leq x: f(n) \geq y\} \]
for all $x$ and $y$. Our first result gives the order of magnitude of the near-practical numbers for essentially all interesting choices of $x$ and $y$.

\begin{thm}\label{thm:uniformpractical} There are absolute positive constants $c_3$ and $c_4$ so that for $4 \leq y \leq x$, we have 
\[ c_3 \frac{x}{\log{y}} \leq N(x,y) \leq c_4 \frac{x}{\log{y}}. \]
\end{thm}

\begin{rmk} To see why the technical restriction $y \geq 4$ is necessary, note that $N(x,x) = 0$ for all $3 < x < 4$. 
\end{rmk}

Theorem \ref{thm:uniformpractical} has the following easy corollary, proved in \S\ref{sec:showstopper}.

\begin{cor}\label{cor:fdensity} For each $m$, the set of natural numbers $n$ with $f(n)=m$ possesses an asymptotic density, say $\rho_m$. The constant $\rho_m$ is positive whenever there is at least one $n$ with $f(n)=m$. Moreover,  \[ \sum_{m=1}^{\infty}\rho_m=1.\]
\end{cor}

We call a natural number $m$ for which $\rho_m$ is nonvanishing (equivalently, an $m$ in the image of $f$) an \emph{additive endpoint}. Thus, Corollary \ref{cor:fdensity} shows that $\rho_m$ is the probability mass function for additive endpoints. The first several additive endpoints are
\[ 1, 3, 7, 12, 15, 28, 31, 39, 42, 56, 60, 63, 73, 90, 91, 96, 100, 104, 108, 112, 120, \dots. \]
Just from this limited data, one might conjecture that $\rho_m$ is usually zero, i.e., zero apart from of a set of $m$ of vanishing asymptotic density. This guess is confirmed, in a much sharper form, in our next theorem.

\begin{thm}\label{thm:showstopper} For each fixed $A> 0$ and all $x\geq 3$, the number of integers in $[1,x]$ which occur as additive endpoints is $\ll_{A} x/(\log{x})^A$. In the opposite direction, the number of additive endpoints up to $x$ exceeds
\[ x/\exp(c_5 (\log\log{x})^3) \]
for all large $x$, for some absolute constant $c_5 > 0$.
\end{thm}

Above, we noted Stewart's result that if $f(n) \geq n-1$, then $f(n)=\sigma(n)$. In this statement, a weak lower estimate on $f(n)$ implies that $f(n)$ is as large as possible. Hausman and Shapiro \cite{HS84} proposed investigating the extent of this curious phenomenon. More specifically, they asked for the slowest-growing monotone function $g(n)$ for which $f(n)\geq g(n)$ implies (at least for $n$ large) that $n$ is practical. Set
\[ HS(n):=\sqrt{e^{\gamma} n\log\log{n}}, \]
where $\gamma$ is the Euler--Mascheroni constant. The next proposition appears as \cite[Theorems 2.1, 2.2]{HS84}.

\begin{prop}\label{prop:HS} Let $\epsilon > 0$. Apart from finitely many exceptional $n$, all solutions to $f(n) \geq (1+\epsilon) HS(n)$ are practical. On the other hand, there are infinitely many non-practical $n$ with $f(n) \geq (1-\epsilon)HS(n)$.\end{prop}
	
Our final theorem removes the factor $1+\epsilon$ from the first half of Proposition \ref{prop:HS}.

\begin{thm}\label{thm:HS} If $n > 3$ and $f(n) \geq HS(n)$, then $n$ is practical.
\end{thm} 

\subsection*{Notation} We use the Landau--Bachmann $o$ and $O$ symbols, as well as Vinogradov's $\ll$ notation, with their usual meanings; subscripts indicate dependence of implied constants. We write $\omega(n):=\sum_{p \mid n}1$ for the number of distinct prime factors of $n$ and $\Omega(n):=\sum_{p^k \mid n} 1$ for the number of prime factors of $n$ counted with multiplicity; $\Omega(n; y) := \sum_{p^k \mid n,~p \leq y} 1$ denotes the number of prime divisors of $n$ not exceeding $y$, again counted with multiplicity. The number of divisors of $n$ is denoted $d(n)$. We use $P^{-}(m)$ for the smallest prime factor of $m$, with the convention that $P^{-1}(1)$ is infinite. Absolute positive constants are denoted by $c_1, c_2, c_3,$ etc., and have the same meaning each time they appear.

\section{Proofs of Theorem \ref{thm:uniformpractical} and Corollary \ref{cor:fdensity}}
We begin by recording some useful lemmas. Our first gives a formula for $f(n)$ in terms of the prime factorization of $n$. 

We assume that the factorization of $n$ has been given in the form \eqref{eq:nfact}. We define $n_0:=1$ and $n_j:=\prod_{1 \leq i \leq j}p_i^{e_i}$. Let $0 \leq j < r$ be the first index for which $p_{j+1} > \sigma(n_j)+1$, putting $j=r$ if no such index exists (i.e., if $n$ is practical). Then $n_j$ is a practical number, by the Stewart--Sierpi\'nski classification, and we call $n_j$ the \emph{practical component} of $n$.
	
\begin{lem}\label{lem:component} We have $f(n)=\sigma(n_j)$, where $n_j$ is the practical component of $n$.
\end{lem}

\begin{proof}  Since $n_j$ is practical, $f(n) \geq f(n_j) = \sigma(n_j)$. On the other hand, $\sigma(n_j)+1$ is not representable as a sum of proper divisors of $n$. Indeed, if $d$ is a divisor of $n$ involved in an additive representation of $\sigma(n_j)+1$, then $d \leq \sigma(n_j)+1 < p_{j+1} < p_{j+2} < \dots < p_r$. It follows that the only primes dividing $d$ are $p_1, \dots, p_j$, so that $d$ is a divisor of $n_j$. But the largest number which can be formed as a sum of distinct divisors of $n_j$ is $\sigma(n_j)$, which is smaller than $\sigma(n_j)+1$. So $\sigma(n_j)+1$ is not representable as a sum of distinct divisors of $n$, and hence $f(n) = \sigma(n_j)$, as claimed.
\end{proof}

The following lemma was observed by Margenstern \cite[Corollaire 1]{margenstern91} to follow from the Stewart--Sierpi\'nski classification.

\begin{lem}\label{lem:margenstern} If $n$ is practical and $m \leq \sigma(n)+1$, then $mn$ is practical.\end{lem}

We now employ Lemma \ref{lem:margenstern} to show that reasonably short intervals contain a positive proportion of practical numbers.

\begin{lem}\label{lem:bertrand} Let $\epsilon > 0$. For $x > x_0(\epsilon)$, the number of practical numbers in $((1-\epsilon)x, x]$ is $\gg_{\epsilon} x/\log{x}$. 
\end{lem}
\begin{proof} We can assume that $0 < \epsilon < 1$. With $c_1$ and $c_2$ as defined in \eqref{eq:saias}, we set $r := \lceil 2c_2/c_1\rceil$ and $s:= \lceil 1/\epsilon\rceil$. From \eqref{eq:saias}, we have that for large $x$ (depending on $\epsilon$),  the number of practical numbers in the interval $(x/rs, x/s]$ is 
\[ \geq c_1 \frac{x/s}{\log{(x/s)}} - c_2 \frac{x/rs}{\log{(x/rs)}} > \frac{c_1}{3s} \frac{x}{\log{x}} \geq \frac{c_1}{6} \epsilon \frac{x}{\log{x}}. \]
By the pigeonhole principle, one of the intervals $(\frac{x}{s+1}, \frac{x}{s}]$, $(\frac{x}{s+2}, \frac{x}{s+1}]$, \dots, $(\frac{x}{rs}, \frac{x}{rs-1}]$ contains $> \frac{c_1}{6rs} \epsilon x/\log{x} \gg \epsilon^2 x/\log{x}$ practical numbers. Suppose this interval is $(\frac{x}{j+1}, \frac{x}{j}]$, where $s \leq j < rs$, and let $n$ be a practical number contained within. If $x > (rs)^2$, then $j < \frac{x}{j+1} < n$, and so $jn$ is practical by Lemma \ref{lem:margenstern}. (Note that the lower bound on $x$ assumed here depends only on $\epsilon$.) Letting $n$ run through the practical numbers in $(\frac{x}{j+1}, \frac{x}{j}]$, we obtain $\gg \epsilon^2 x/\log{x}$ practical numbers $jn \in (x\frac{j}{j+1},x]$. But $(x\frac{j}{j+1}, x] \subset ((1-\epsilon)x, x]$, by our choice of $s$. This proves Lemma \ref{lem:bertrand}. Moreover, we have shown that the implied constant in the lemma statement may be taken proportional to $\epsilon^2$.
\end{proof}

The next result, due to Hausman and Shapiro \cite[Theorem 4.1]{HS84}, shows that substantially shorter intervals than those considered in Lemma \ref{lem:bertrand} always contain at least one practical number.

\begin{lem}\label{lem:veryshort} For all real $x\geq 1$, there is a practical number $x < n < x + 2x^{1/2}$. 
\end{lem}

Let $\Phi(x,y)$  denote the number of natural numbers $n\leq x$ divisible by no primes $\leq y$. The following lemma is a consequence of Brun's sieve. Variants can be found, e.g., as \cite[Theorem 1, p. 201]{HR83} or \cite[Theorem 3, p. 400]{tenenbaum95}.

\begin{lem}\label{lem:sieve} Uniformly for $2 \leq y \leq x$, we have $\Phi(x,y) \ll x/\log{y}$. If we assume also that $x > c_6 y$ for a suitable large absolute constant $c_6$, then $\Phi(x,y) \gg x/\log{y}$.
\end{lem}

We now prove Theorem \ref{thm:uniformpractical}, treating the upper and lower estimates separately.

\begin{proof}[Proof of the upper bound in Theorem \ref{thm:uniformpractical}] Suppose that $n \leq x$ and $f(n)\geq y$. By the upper bound in \eqref{eq:saias}, we may restrict our attention to non-practical $n$. Let $d$ be the practical component of $n$ and write $n=dq$. By Lemma \ref{lem:component}, $\sigma(d) = f(n)$. In particular, since we are assuming that $f(n)\geq y\geq 4$, we must have that $d > 1$. Moreover, since $n$ is not practical, $d < n$. Thus, $q > 1$ and \[ P^{-}(q) > \sigma(d)+1 > d.\] Hence, \[ d^2 < d \cdot P^{-}(q) \leq dq = n \leq x,\] and so $d \leq \sqrt{x}$. 
	
Given $d$, the number of possibilities for $n$ is bounded above by the number of $q \leq x/d$ with $P^{-}(q) > d$. Since $2 \leq d \leq x/d$, we may apply Lemma \ref{lem:sieve} to find that the number of possibilities for $q$ is $\ll \frac{x}{d\log{d}}$. Since $\sigma(d) = f(n) \geq y$ and (crudely) $\sigma(d) < d^2$, it follows that $d > \sqrt{y}$. Hence, using partial summation and \eqref{eq:saias}, we see that the number of possibilities for $n$ is 
\begin{align*} \ll x\sum_{\substack{\sqrt{y} < d \leq \sqrt{x}\\d\text{ practical}}} \frac{1}{d\log{d}}&\leq x\frac{PR(\sqrt{x})}{\sqrt{x} \log\sqrt{x}} + x\int_{\sqrt{y}}^{\sqrt{x}} PR(t) \frac{1+\log{t}}{(t\log{t})^2}\, dt \\
	&\ll \frac{x}{(\log{x})^2} + x\int_{\sqrt{y}}^{\sqrt{x}}\frac{dt}{t (\log{t})^2} \ll \frac{x}{(\log{x})^2} + \frac{x}{\log{y}}\ll \frac{x}{\log{y}}. \qedhere\end{align*}
\end{proof}

\begin{proof}[Proof of the lower bound in Theorem \ref{thm:uniformpractical}] The proof is suggested by that offered for the upper bound, but some care is necessary to ensure uniformity throughout the stated range of $x$ and $y$. 
	
First, we treat the range when $x^{1/10}\leq y\leq x$. In this domain, we use the trivial lower bound $N(x,y) \geq N(x,x)$. We estimate the right-hand side from below by counting practical numbers $n$ belonging to the interval $[\frac{x+1}{2},x]$. Note that for such $n$, we have $f(n)=\sigma(n) \geq 2n-1 \geq x$  (using for the first inequality that $n-1$ is a sum of proper divisors of $n$), and so $n$ is indeed counted by $N(x,x)$. 

If $6 \leq x \leq 11$, then $n=6$ is a practical number in $[\frac{x+1}{2},x]$. Similarly, if $4 \leq x \leq 6$, then $n=4$ works. Finally, if $x\geq 11$, then Lemma \ref{lem:veryshort} gives a practical number $n$ with
\[ \frac{x+1}{2} < n < \frac{x+1}{2} + 2\sqrt{\frac{x+1}{2}} \leq x. \]
Hence, we always have $N(x,x)\geq 1$. (Recall that we only consider $x\geq 4$.) Moreover, by Lemma \ref{lem:bertrand}, there are $\gg x/\log{x}$ practical numbers in $[\frac{x+1}{2},x]$ once $x$ is large. It follows that $N(x,x) \gg x/\log{x}$ for all $x\geq 4$. So if $x^{1/10} \leq y\leq x$, then
\[ N(x,y) \geq N(x,x) \gg x/\log{x} \gg x/\log{y}, \]
which gives the lower bound of the theorem in this case.

Now suppose that $4 \leq y \leq x^{1/10}$. We consider numbers of the form $n=dq\leq x$, where $d$ is a practical number in $(y,y^3]$ and where $P^{-}(q) > y^6$. For any such $n$, we have $f(n) \geq f(d) \geq d > y$. Moreover, each $n$ constructed in this way arises exactly once, since $q$ is determined as the largest divisor of $n$ supported on primes $> y^6$. Given $d$, the number of corresponding $q$ is $\Phi(x/d,y^6)$. If $x$ is large, then 
\[ \frac{x/d}{y^6} \geq \frac{x}{y^9} \geq x^{1/10} > c_6,\]
and so Lemma \ref{lem:sieve} gives
\begin{equation}\label{eq:weget} \Phi(x/d,y^6) \gg \frac{x}{d\log{y}}. \end{equation}
On the other hand, \eqref{eq:weget} is trivial for bounded $x$, since $1$ is always counted by $\Phi(x/d,y^6)$. Thus, \eqref{eq:weget} holds in any case. Hence, the number of $n$ constructed in this way is
\[ \gg \frac{x}{\log{y}}\sum_{\substack{y < d \leq y^3\\d\text{ practical}}}\frac{1}{d}. \] 
That the sum appearing here is $\gg 1$ for large $y$ follows from partial summation and the lower bound in \eqref{eq:saias}. For bounded $y$, the sum is also $\gg 1$, since Lemma \ref{lem:veryshort} guarantees that there is at least one practical number between $y$ and $y^3$. (Certainly $y + 2y^{1/2} < 3y < y^3$ when $y \geq 4$.) This completes the proof of the lower bound.
\end{proof}

\begin{proof}[Proof of Corollary \ref{cor:fdensity}] One can detect whether or not $f(n) = m$ given just the list of divisors of $n$ not exceeding $m+1$. Thus, whether or not $f(n) = m$ depends only on the residue class of $n$ modulo $(m+1)!$. This gives the first two assertions of the corollary. For the third, notice that $1-\sum_{n=1}^{N}\rho_m$ represents the density of the set of $n$ with $f(n) > N$, which is $\ll 1/\log{N}$ by Theorem \ref{thm:uniformpractical}. Letting $N\to\infty$ completes the proof.\end{proof}

\section{Proof of Theorem \ref{thm:showstopper}}\label{sec:showstopper}

We divide the proof of Theorem \ref{thm:showstopper} into two parts.

\subsection{The upper bound in Theorem \ref{thm:showstopper}}
Central to the proof of both halves of Theorem \ref{thm:showstopper} is the observation, immediate from Lemma \ref{lem:component}, that $m$ belongs to the range of $f$ precisely when $m = \sigma(n)$ for some practical number $n$. Thus, we are really asking  in Theorem \ref{thm:showstopper} for estimates on the range of $\sigma$ restricted to practical inputs.

\begin{lem}\label{lem:bigorsmall} Let $A\geq 30$. Suppose that $x \geq 3$. If $n$ is a practical number with $x^{3/4} < n \leq x$, then either
	\begin{equation}\label{eq:bigOmega} \Omega(n) > 2A\log\log{x} \end{equation}
	or
	\begin{equation}\label{eq:littleomega} \omega(n) > \frac{1}{2\log{A}} \log\log{x}. \end{equation}
\end{lem}
\begin{proof} Since $n$ is practical, every integer in $[1,n]$ can be written as a subset-sum of divisors of $n$. Thus, $2^{d(n)} \geq n$, so we can use the hypothesis that $n > x^{3/4}$ to show \[ d(n) \geq \frac{\log{n}}{\log{2}} > \frac{3/4}{\log{2}} \log{x} > \log{x}.\]
Suppose that $n = \prod_{i=1}^{l}p_i^{e_i}$ is the factorization of $n$ into primes, where $l=\omega(n)$. Since $d(n) = \prod_{i=1}^{l}(e_i+1) > \log{x}$, the inequality between the arithmetic and geometric means gives that
\begin{equation}\label{eq:AMGM} \frac{1}{l^l}\left(\sum_{i=1}^{l}(e_i+1)\right)^l \geq \prod_{i=1}^{l}(e_i+1) > \log{x}. \end{equation}

Now assume that \eqref{eq:bigOmega} fails. Then $\sum_{i=1}^{l}(e_i+1) \leq 2 \sum_{i=1}^{l}e_i \leq 4A\log\log{x}$,  and \eqref{eq:AMGM} gives $(\frac{4A\log\log{x}}{l})^l > \log{x}$. Writing $l = \lambda \log\log{x}$, we deduce that
\[ \left(\frac{4A}{\lambda}\right)^{\lambda \log\log{x}} > \log{x}, \quad\text{and so}\quad \lambda \log \frac{4A}{\lambda} > 1. \]
This latter inequality, along with the condition $A \geq 30$, implies that $\lambda > \frac{1}{2\log{A}}$ (by a short exercise in calculus). Since $\omega(n)= \lambda \log\log{x}$, we have \eqref{eq:littleomega}.\end{proof}

The next lemma, which belongs to the study of the anatomy of integers, bounds from above the number of $n$ with an abnormally large number of small prime factors.

\begin{lem}\label{lem:HT} Let $x, y \geq 2$, and let $k \geq 1$. The number of $n \leq x$ with $\Omega(n;y) \geq k$ is $\ll \frac{k}{2^k} x \log{y}$.
\end{lem}

\begin{rmk} As a special case (when $y=x$), the number of $n \leq x$ with $\Omega(n) \geq k$ is $\ll \frac{k}{2^k}x\log{x}$.
\end{rmk} 

\begin{proof} The proof is almost identical to that suggested in Exercise 05 of \cite[p. 12]{HT88}, details of which can be found in \cite[Lemmas 12, 13]{LP07}. Thus, we only sketch it. Let $v:=2-1/k$. Let $g$ be the arithmetic function determined through the convolution identity $v^{\Omega(n;y)}= \sum_{d \mid n}g(d)$. Then $g$ is multiplicative. For $e\geq 1$, we have $g(p^e) = v^e-v^{e-1}$ if $p \leq y$, and $g(p^e) = 0$ if $p > y$. Hence,
\begin{align*} \sum_{n \leq x}v^{\Omega(n; y)} &= \sum_{d \leq x}g(d)\left\lfloor \frac{x}{d}\right\rfloor \leq x \sum_{d\leq x}\frac{g(d)}{d} \\&\leq x\prod_{p \leq y}\left(1 + \frac{v-1}{p} + \frac{v^2-v}{p^2}+ \dots\right) = \frac{x}{2-v}\prod_{3 \leq p \leq y}\left(1 + \frac{v-1}{p-v}\right).
\end{align*}
Now $2-v=1/k$, and the rightmost product is
\[ \leq \exp\left(\sum_{3 \leq p \leq y}\frac{v-1}{p-v}\right) \leq \exp\left(\sum_{3 \leq p \leq y}\frac{1}{p-2}\right) \leq \exp\left(\sum_{p \leq y}\frac{1}{p} + O(1)\right) \ll \log{y}.\] Collecting our estimates, we have shown that
\[ \sum_{n \leq x} v^{\Omega(n;y)} \ll k x\log{y}. \]
But each term with $\Omega(n;y) \geq k$ makes a contribution to the left-hand side that is $\geq v^k \geq (2-1/k)^k = 2^k (1-\frac{1}{2k})^k \gg 2^k$. Thus, the number of such terms is $\ll \frac{k}{2^k} x \log{y}$.
\end{proof}

The next lemma can be viewed as a partial shifted-primes analogue of the Hardy--Ramanujan inequalities. A proof can be found in the text of Prachar \cite[Lemma 7.1, p. 166]{prachar57} (cf. Erd\H{o}s \cite{erdos35}). There a slightly stronger assertion is shown for shifted primes $p-1$; only trivial changes are required to replace $p-1$ with $p+1$.

\begin{lem}\label{lem:EP} Let $t \geq 3$, and let $k \geq 1$. The number of primes $p\leq t$ with $\omega(p+1)=k$ is 
	\[ \ll \frac{t}{(\log{t})^2} \left(\frac{(\log\log{t} + c_7)^{k+2}}{(k-1)!} + 1\right). \]
\end{lem}
	
\begin{proof}[Proof of the upper bound in Theorem \ref{thm:showstopper}] It is enough to prove the result for large values of $A$. Suppose that $m\leq x$ is an additive endpoint, and write $m = \sigma(n)$ with $n$ practical. Put $Z := 2A\log\log{x}$. The number of values of $m$ corresponding to an integer $n \leq x^{3/4}$ or an $n$ with $\Omega(n) > Z$ is, by Lemma \ref{lem:HT}, 
	\[ \ll x^{3/4} + \frac{Z}{2^{Z}} x \log{x} \ll_{A} \frac{x}{(\log{x})^A}. \]
Thus, with \[ Z' := \frac{1}{2\log{A}} \log\log{x},\] Lemma \ref{lem:bigorsmall} allows us to assume that
\begin{equation}\label{eq:lotsaprimes} \omega(n) \geq Z'. \end{equation}

We now show that most of the primes dividing $n$ make a large contribution to $\Omega(\sigma(n)) = \Omega(m)$. We claim we can assume that both of the following hold:
\begin{enumerate} 
\item There are fewer than $Z'/4$ primes $p$ for which $p^2 \mid n$.
\item There are fewer than $Z'/4$ primes $p$ dividing $n$ for which
\begin{equation}\label{eq:largeshifted} \Omega(p+1) \leq 16 A\log{A}. \end{equation}
\end{enumerate}
With $K:= \lceil Z'/4\rceil$, the number of $n\leq x$ which are exceptions to (i) is, by the multinomial theorem,
\[ \leq x \sum_{\substack{d \leq x, \text{ squarefree}\\ \omega(d) =K}}\frac{1}{d^2} \leq \frac{x}{K!}\left(\sum_{p \leq x}\frac{1}{p^2}\right)^K \leq  x (e/K)^{K} < x/(\log{x})^A, \]
once $x$ is large. (We use here that $\sum{p^{-2}} < 1$ and the elementary inequality $K! \geq (K/e)^k$.) To handle (ii), we observe that from Lemma \ref{lem:EP} and partial summation, the sum of the reciprocals of all $p$ satisfying \eqref{eq:largeshifted} converges. Let $S$ denote this sum. Then the number of exceptions to (ii) is, for large $x$,
\[ \leq \frac{x}{K!}\Bigg(\sum_{\substack{p \leq x\\ p\text{ satisfies \eqref{eq:largeshifted}}}}\frac{1}{p}\Bigg)^K \leq  x (eS/K)^{K} < x/(\log{x})^A. \]
Hence, we can indeed assume (i) and (ii). 

From \eqref{eq:lotsaprimes}, it now follows that there are at least $Z' - 2\frac{Z'}{4} = \frac{Z'}{4}$ primes $p$ for which $p \parallel n$ and for which $\Omega(p+1) > 16 A\log{A}$. Hence,
\[ \Omega(m) = \Omega(\sigma(n)) \geq \sum_{p\parallel n} \Omega(p+1) > 16 A\log{A} \cdot \frac{Z'}{4} = 2 A\log\log{x}. \]
But by Lemma \ref{lem:HT}, the number of $m \leq x$ with $\Omega(m)$ this large is $\ll_{A} x/(\log{x})^A$. This completes the proof of Theorem \ref{thm:showstopper} for large $x$. If $x$ is bounded in terms of $A$, then the theorem is trivial.
\end{proof}

\begin{rmk} The method given here can be pushed to yield the more explicit result that the count of $m\leq x$ that occur as additive endpoints is smaller than
\[ x/\exp\left(c_8 \log\log{x} \frac{\log\log\log{x}}{\log\log\log\log{x}}\right). \]
\end{rmk}

\subsection{The lower bound in Theorem \ref{thm:showstopper}}
The lower bound in Theorem \ref{thm:showstopper} will be deduced from the following proposition, which may be of interest outside of this context. 

\begin{prop}\label{prop:thickness} Let $A> 0$. There is a constant $c=c(A)$ so that the following holds. If $x$ is sufficiently large, say $x > x_0(A,c)$, then any subset $\Ss \subset [1,x]$ with
	\[ \#\Ss \leq x/\exp(c (\log\log{x})^3) \]
	satisfies
	\[ \#\sigma^{-1}(\Ss) \leq x/(\log{x})^A. \]
	Here $\sigma^{-1}(\Ss)$ denotes the set of $n$ with $\sigma(n) \in \Ss.$
\end{prop}

\begin{rmk} It is perhaps surprising that one cannot improve the upper bound on $\#\sigma^{-1}(\Ss)$ very much, even if one assumes that $\Ss$ consists of only a single element! Indeed, plausible conjectures about the distribution of smooth shifted primes $p+1$ (such as what would follow from the Elliott--Halberstam conjecture) imply that for all large $x$, there is a singleton set $\Ss \subset [1,x]$ with $\#\sigma^{-1}(\Ss) > x^{1-\epsilon}$. (Here $\epsilon > 0$ is arbitrary but fixed.) For the Euler $\phi$-function, this result is due to Erd\H{o}s \cite{erdos35} (see also the exposition of Pomerance \cite{pomerance89}); the $\sigma$-version can be proved similarly, replacing $p-1$ with $p+1$ when necessary.
\end{rmk}

To apply Proposition \ref{prop:thickness} to the case of the practical numbers, it is convenient to recall Gronwall's determination of the maximal order of the sum-of-divisors function $\sigma$ \cite[Theorem 323, p. 350]{HW08}. 

\begin{lem}\label{lem:eulerphi} We have $\limsup_{n\to\infty} \frac{\sigma(n)}{n\log\log{n}} = e^{\gamma}$.
\end{lem}

\begin{proof}[Proof of the lower bound in Theorem \ref{thm:showstopper}]
Let $x$ be large. By Lemma \ref{lem:eulerphi}, if $n \leq \frac{x}{2\log\log{x}}$, then $\sigma(n) \leq x$. (We use here that $e^{\gamma} < 2$.) Thus, with $\Ss$ the set of additive endpoints not exceeding $x$, 
\[ \#\sigma^{-1}(\Ss) \geq PR\left(\frac{x}{2\log\log{x}}\right) \gg \frac{x}{(\log{x})(\log\log{x})}, \]
using the lower estimate in \eqref{eq:saias} for the last step. The desired lower bound on $\#\Ss$ now follows from (the contrapositive of) Proposition \ref{prop:thickness}, with $A=1.1$.
\end{proof}

The rest of this section is devoted to the proof of Proposition \ref{prop:thickness}. The proof rests on a $\sigma$-analogue of a result for the Euler function appearing in a paper of Luca and the first author \cite[Lemma 2.1]{LP11}. 

\begin{lem}\label{lem:sigmadivides} Let $x\geq 3$. Let $d$ be a squarefree natural number with $d\leq x$. The number of $n$ for which $d \mid \sigma(n)$ and $\sigma(n) \leq x$ is 
	\[ \leq \frac{x}{d}(c_9 \log{x})^{3\omega(d)}. \]
\end{lem}
\begin{proof} If $d=1$, the result is clear. Suppose that $d > 1$. Let $n$ be an integer for which $\sigma(n) \in [1,x]$ is a multiple of $d$. Write the prime factorization of $n$ in the form $n = \prod_{i}p_i^{e_i}$. Since $d \mid \sigma(n)$, there is a factorization $d= d_1 d_2 \cdots$ for which each $d_i \mid \sigma(p_i^{e_i})$. Discarding those terms with $d_i=1$ and relabeling, we can assume that $d=d_1 \cdots d_l$, where each $d_i > 1$. Clearly, $l \leq \omega(d)$. 
	
We now fix the factorization $d=d_1\cdots d_l$ and count the number of corresponding $n$. This count does not exceed 
\begin{equation}\label{eq:inserthere} x \prod_{i=1}^{l} \left(\sum_{\substack{p^e:~\sigma(p^e)\leq x\\ d_i \mid \sigma(p^e)}} \frac{1}{p^e}\right). \end{equation}
We proceed to estimate the inner sum in \eqref{eq:inserthere}. If $d_i \mid \sigma(p^e)$, then $\sigma(p^e) = d_i m$, with $m \leq x/d_i$. Since $\sigma(p^e) = 1 + p + \dots +p^e \leq 2p^e$, 
\[ \sum_{\substack{p^e:~\sigma(p^e)\leq x\\ d_i \mid \sigma(p^e)}} \frac{1}{p^e} \leq \frac{2}{d_i}\sum_{m \leq x/d_i} \frac{1}{m}\sum_{p^e\colon~ \sigma(p^e)=md_i} 1. \]
For each fixed $e\geq 1$, there is at most one prime $p$ with $\sigma(p^e)=md_i$; moreover, since $md_i \leq x$, there are no such $p$ once $e> \log{x}/\log{2}$. Thus, 
\[ \frac{2}{d_i}\sum_{m \leq x/d_i} \frac{1}{m}\sum_{p^e\colon~ \sigma(p^e)=md_i} 1 \ll \frac{\log{x}}{d_i} \sum_{m \leq x/d_i}\frac{1}{m} \ll \frac{(\log{x})^2}{d_i}. \]
Inserted back into \eqref{eq:inserthere}, we find that for a certain absolute constant $C > 1$, the number of $n$ corresponding to the given factorization is at most
\[ x \prod_{i=1}^{l} \frac{C(\log{x})^2}{d_i} = \frac{x}{d} C^{l} (\log{x})^{2l} \leq \frac{x}{d} C^{\omega(d)} (\log{x})^{2\omega(d)}. \]

Finally, we sum over unordered factorizations of $d$ into parts $> 1$. Since $d$ is squarefree, the number of such factorizations is precisely $B_{\omega(d)}$, where $B_k$ denotes the $k$th Bell number (the number of set partitions of a $k$-element set). Thinking combinatorially, we have the crude bound $B_k \leq k^k$, and so the total number of $n$ which arise is
\[ \leq \omega(d)^{\omega(d)}\left( \frac{x}{d} C^{\omega(d)} (\log{x})^{2\omega(d)}\right) = \frac{x}{d}(C \omega(d) (\log{x})^2)^{\omega(d)}. \]
By definition, we have $\omega(d) \leq \Omega(d) \leq \log{x}/\log{2}$, where the final inequality follows from the simple observation that $2^{\Omega(d)} \leq d \leq x$. This proves our lemma with $c_9 = (C/\log{2})^{1/3}$.
\end{proof}

\begin{lem}\label{lem:sigmaOmega} Fix $A \geq 3$. The number of $n \leq x$ for which
	\begin{equation}\label{eq:Omegasigma} \Omega(\sigma(n)) \geq 8 A^2 (\log\log{x})^2 \end{equation}
	is $o(x/(\log{x})^A)$, as $x\to\infty$.
\end{lem}

\begin{proof} We may suppose that $\omega(n) \leq 2A\log\log{x}$. Indeed, Lemma \ref{lem:HT} shows that for $x\geq 3$, the number of $n \leq x$ not satisfying the stronger inequality $\Omega(n) \leq 2A\log\log{x}$ is 
\[ \ll \frac{A \log\log{x}}{2^{A\log\log{x}}} x\log{x} \ll_{A} \frac{x\log\log{x}}{(\log{x})^{2A\log{2}-1}}. \]
Since $A\geq 3$, the exponent $2A\log{2}-1 > A$, and so this upper bound is $o(x/(\log{x})^A)$. 

Writing $\Omega(\sigma(n)) = \sum_{p^e\parallel n}\Omega(\sigma(p^e))$, we thus deduce that if \eqref{eq:Omegasigma} holds, then 
	\begin{equation}\label{eq:Omegasigma2} \Omega(\sigma(p^e)) \geq \frac{8A^2 (\log\log{x})^2}{2A\log\log{x}} = 4A \log\log{x} \end{equation}
for some prime power $p^e \parallel n$. 

Suppose first that $e > 1$. Then (for large $x$) the squarefull part of $n$ is of size at least
	\begin{equation}\label{eq:power5A} p^e \geq \frac{1}{2}\sigma(p^e) \geq \frac{1}{2} 2^{\Omega(\sigma(p^e))} \geq \frac{1}{2} 2^{4A\log\log{x}} > (\log{x})^{5A/2}. \end{equation}

But then the number of possibilities for $n\leq x$ is $\ll x/(\log{x})^{5A/4}$, and so in particular is $o(x/(\log{x})^A)$.  On the other hand, if $e=1$, then \eqref{eq:Omegasigma2} implies that $n$ is divisible by some prime $p$ with $\Omega(p+1) \geq 4A\log\log{x}$. For each such $p$, the number of corresponding $n$ is $\leq x/p < 2x/(p+1)$. Summing over $p$,  we find that the total number of such $n\leq x$ is at most
\[ 2x\sum_{\substack{d \leq x \\ \Omega(d) \geq 4A\log\log{x}}}\frac{1}{d}. \]
Put $Z:=4A\log\log{x}$; by partial summation, along with Lemma \ref{lem:HT} and the final inequality in \eqref{eq:power5A}, this upper bound is
\[ \ll x \frac{Z}{2^Z}\int_2^{x}\frac{\log{t}}{t}\, dt \ll x (\log{x})^2 \frac{Z}{2^Z} \ll_{A} x\frac{(\log{x})^2 \log\log{x}}{(\log{x})^{5A/2}}, \]
and so is $o(x/(\log{x})^{A})$, as $x\to\infty$. This completes the proof.
\end{proof}

\begin{lem}\label{lem:sigmaRepeated} Let $x \geq 3$, and let $z \geq 1$. The number of $n \leq x$ with $\sigma(n)$ divisible by $p^2$ for some prime $p > z$ is $\ll x(\log{x})^2 z^{-1/2}$.
\end{lem}
\begin{proof} If $p^2 \mid \sigma(n)$, then either $p \mid \sigma(q^e)$ for a proper prime power $q^e$ exactly dividing $n$, or there are two distinct primes $q_1$ and $q_2$ exactly dividing $n$ with $q_1, q_2 \equiv -1\pmod{p}$. In the former case, $n$ has a squarefull divisor of size $\geq q^e \geq \frac{1}{2}\sigma(q^e) > p/2 > z/2$. The number of such $n$ is $\ll x z^{-1/2}$, which is acceptable for us. For a given $p$, the number of $n$ arising in the second case is
	\[ \leq x \Bigg(\sum_{\substack{q \leq x \\ q\equiv -1\pmod{p}}}\frac{1}{q}\Bigg)^2 \leq x \left(\sum_{j \leq x/p}\frac{1}{pj-1}\right)^2 \ll x (\log{x})^2 p^{-2}.  \]
Summing over $p > z$, we find that the total number of $n$ that can arise from this case is $\ll x(\log{x})^2 z^{-1}$, which is also acceptable.
\end{proof}

\begin{proof}[Proof or Proposition \ref{prop:thickness}] We may suppose that our fixed constant $A$ satisfies $A\geq 5$. We will show that for such $A$, the proposition holds with $c(A) = 50 A^3$. 
	
Let $\Ss_1$ consist of those $m \in \Ss$ for which either
\begin{enumerate}
	\item $m \leq x/(\log{x})^{2A}$, or
\item $\Omega(m) \geq 8A^2 (\log\log{x})^2$, or
\item $p^2 \mid m$ for some $p > (\log{x})^{3A}$.
\end{enumerate}
We let $\Ss_2$ consist of the remaining elements of $\Ss$. By Lemmas \ref{lem:sigmaOmega} and \ref{lem:sigmaRepeated}, the size of $\sigma^{-1}(\Ss_1)$ is $o(x/(\log{x})^A)$ as $x\to\infty$ (uniformly in the choice of $\Ss$). 

We turn now to $\Ss_2$. To each $m \in \Ss_2$, we associate the divisor $m'$ of $m$ defined by 
\[ m' := \prod_{\substack{p^e \parallel m \\ p > (\log{x})^{3A}}}p. \]
Then $m'$ is squarefree, and \begin{equation}\label{eq:omegam'} \omega(m') \leq \Omega(m) < 8A^2 (\log\log{x})^2. \end{equation} Moreover, assuming that $x$ is large, since $m > (x/\log x)^{2A}$,
\begin{equation}\label{eq:omegam'2} m' \geq m/((\log{x})^{3A})^{\Omega(m)} > \frac{x/(\log{x})^{2A}}{\exp(24A^3 (\log\log{x})^3)} > x/\exp(25 A^3 (\log\log{x})^3). \end{equation}
We bound the number of $\sigma$-preimages of $m$ from above by the number of $n$ for which $\sigma(n) \in [1,x]$ is a multiple of $m'$. By Lemma \ref{lem:sigmadivides}, along with \eqref{eq:omegam'} and \eqref{eq:omegam'2}, the number of such $n$ is
\begin{align*} \leq \frac{x}{m'} (c_9\log{x})^{3\omega(m')} &\leq \exp(25A^3(\log\log{x})^3) (c_9 \log{x})^{24A^2 (\log\log{x})^2}  \\
	&\leq \exp(49A^3 (\log\log{x})^3),
\end{align*}
say. Summing over the elements of $\Ss_2$, we find that
\[ \#\sigma^{-1}(\Ss_2) \leq \exp(49A^3 (\log\log{x})^3) \cdot \#\Ss_2. \]
So if we assume that $\#\Ss \leq x/\exp(50A^3 (\log\log{x})^3)$, then  $\#\sigma^{-1}(\Ss_2) = o(x/(\log{x})^A)$, as $x\to\infty$. Combined with our earlier estimate on the size of $\sigma^{-1}(\Ss_1)$, this shows that $\#\sigma^{-1}(\Ss) \leq x/(\log{x})^A$ once $x$ is sufficiently large.
\end{proof}

\section{Proof of Theorem \ref{thm:HS}}

The key to the proof of Theorem \ref{thm:HS} is the following inequality of Robin \cite[Th\'eor\`{e}me 2]{robin84}.

\begin{lem}\label{lem:robin} For each natural number $n \geq 3$, 
	\[ \sigma(n) \leq e^{\gamma} n\log\log{n} + 0.6483 \frac{n}{\log\log{n}}. \]
\end{lem} 

\begin{proof}[Proof of Theorem \ref{thm:HS}] Suppose for the sake of contradiction that $f(n) \geq HS(n)$ but that $n$ is not practical. We assume to begin with that $n > 14$, treating small $n$ at the end of the proof. Let $d$ be the practical component of $n$, and write $n=dq$. Then $q > 1$, and 
\[ P^{-}(q) > \sigma(d)+1 = f(n)+1 > HS(n) > n^{1/2}, \]
where in fact the last inequality holds for all $n > 6$. It follows that $q$ is prime and $P^{-}(q)=q$. Hence, $HS(n) < q = n/d$, and so
\[ d < \frac{n}{HS(n)}. \]
Also, since $\sigma(d) =f(n) \geq HS(n)$, we have
\[ q = n/d \leq \frac{n}{d} \frac{\sigma(d)}{HS(n)}. \]
Multiplying the last two displayed inequalities shows that
\[ n =dq \leq \frac{\sigma(d)}{d}\left(\frac{n}{HS(n)}\right)^2 = \frac{\sigma(d)}{d}n (e^{\gamma}\log\log{n})^{-1},\]
and so 
\begin{equation}\label{eq:contrary}
 \frac{\sigma(d)}{d} \geq e^{\gamma}\log\log{n}.\end{equation}
Since $q > n^{1/2}$ and $n=dq$, we have that $q > d$, and so
\[ \log\log{n} = \log\log{(qd)} > \log\log{(d^2)} = \log\log{d} + \log{2}; \]
thus, \eqref{eq:contrary} gives
\begin{align}\notag \frac{\sigma(d)}{d} &\geq e^{\gamma}\log\log{d} + e^{\gamma} \log{2} \\ 
	\label{eq:contrary2} &> e^{\gamma}\log\log{d}+1.2345. \end{align}
We now derive a contradiction to Robin's inequality. We can assume that $d \geq 6$; otherwise, $\sigma(d)/d \leq 7/4$, and \eqref{eq:contrary} then implies that $n \leq 14$, contrary to hypothesis. By Lemma \ref{lem:robin},
\[ \frac{\sigma(d)}{d} \leq e^{\gamma}\log\log{d} + \frac{0.6483}{\log\log{d}}. \]
Combining this inequality with \eqref{eq:contrary2}, we obtain $0.6483/\log\log{d} > 1.2345$. But this fails for all $d \geq 6$. This contradiction completes the proof for $n > 14$.

It remains to treat the cases when $3 < n \leq 14$. For odd $n > 3$, the hypotheses of the theorem are never satisfied, since $f(n)=1 < HS(5) \leq HS(n)$. So the only possible exceptions to the theorem have $n$ even. The non-practical even values of $n \leq 14$ are $n=10$ and $n=14$, and in both cases, $f(n)=3 < HS(n)$, so the theorem holds.
\end{proof}

\section*{Acknowledgements}
We thank Greg Martin and Carl Pomerance for helpful conversations.

\providecommand{\bysame}{\leavevmode\hbox to3em{\hrulefill}\thinspace}
\providecommand{\MR}{\relax\ifhmode\unskip\space\fi MR }
\providecommand{\MRhref}[2]{%
  \href{http://www.ams.org/mathscinet-getitem?mr=#1}{#2}
}
\providecommand{\href}[2]{#2}

\end{document}